\documentclass[11pt]{amsart}
\usepackage{amsthm, amssymb, latexsym, amsmath, color}
\usepackage[a4paper, total={6.5in, 8in}]{geometry}
\usepackage[dvipsnames,svgnames,table]{xcolor}
\usepackage[colorlinks=true,linkcolor=RoyalBlue,urlcolor=RoyalBlue,citecolor=PineGreen]{hyperref}
\usepackage{tikz}
\usetikzlibrary{angles,quotes}
\usepackage{tikz-cd} 
\usetikzlibrary{arrows.meta}
\usepackage{comment, caption}
\usepackage{graphicx,subcaption}
\usepackage{float}

\usepackage{pgf,tikz,pgfplots}
\pgfplotsset{compat=1.14}
\usepackage{pgf,tikz,pgfplots}
\usepackage{mathrsfs}
\usetikzlibrary{arrows}
\usepackage{mathrsfs}
\usetikzlibrary{arrows}
\usepackage{capt-of}
\usetikzlibrary{decorations.pathreplacing}
\usetikzlibrary{cd}
\usetikzlibrary{positioning}
\usetikzlibrary{calc}
\usepackage{algorithmic}
\usepgfplotslibrary{fillbetween}
\usetikzlibrary{intersections}
\usetikzlibrary{patterns}

\DeclareMathOperator{\rank}{rank}

\usepackage[nameinlink]{cleveref}

\title{Point sets determining few angles are almost contained in a line or circle}
 \author{Krishnendu Bhowmick}
 \address{Krishnendu Bhowmick \\ School of Electrical \& Computer Engineering \\ Tel Aviv University, Israel.}
 \email{krish.combi@gmail.com}

  \author{Oliver Roche-Newton}
 \address{Oliver Roche-Newton \\ Institute of Analysis and Number Theory, TU Graz, Austria.}
 \email{o.rochenewton@gmail.com}

 \author{Audie Warren}
 \address{Audie Warren \\ Johann Radon Institute for Computational and Applied Mathematics, Linz, Austria.}
 \email{audie.warren@oeaw.ac.at}

\newcommand{\cA}{\mathcal{A}}

\newcommand{\cD}{\mathcal{D}}

\newtheorem{lemma}{Lemma}

\newtheorem{theorem}{Theorem}

\newtheorem{construction}{Construction}

\theoremstyle{remark}

\begin{document}

 \begin{abstract}
   We prove a structural theorem for point sets in $\mathbb R^2$ which determine few pinned angles. More precisely, we prove the existence of an absolute constant $c>0$ such that if $n$ is sufficiently large and $P$ is a set of $n$ points then there exists a point $q \in P$ which determines at least $n^{1+c}$ distinct angles to other pairs of points of $P$, provided that $P$ is not of one of the following exceptional forms:
   \begin{itemize}
       \item all but at most one of the points of $P$ lie on a line,
       \item all but two points of $P$ lie on a line, and the two exceptional points are symmetric with respect to the line,
       \item all the points of $P$ lie on a circle,
       \item all but one of the points of $P$ lie on a circle, and the exceptional point is the centre of the circle.
   \end{itemize}
   As a consequence, we answer a question of Corr\'{a}di, Erd\H{o}s and Hajnal by showing that if $n$ is sufficiently large and $P \subseteq \mathbb R^2$ has cardinality $n$ and is not contained on a single line, then $P$ determines at least $n-2$ angles. Moreover, we prove that the unique point set attaining this minimum is the regular $n$-gon.
 \end{abstract}

\maketitle
\section{Introduction}

This paper considers the problem of lower bounding the number of angles determined by a set of $n$ points in the plane. For three distinct points $p,q,r \in \mathbb R^2$, let $\mathcal A(p,q,r) \in [0, \pi]$ denote the angle determined by these three points with $q$ as the centre. For a finite set $P \subset \mathbb R^2$, the notation
\[
\mathcal A(P):= \{ \mathcal A(p,q,r) : p,q,r \in P \text{ and $p,q$ and $r$ are distinct} \}.
\]
is used for the set of angles determined by $P$. We often consider the set of angles pinned at a fixed centre $q \in \mathbb R^2$, for which we use the notation
\[
\mathcal A_q(P):= \{ \mathcal A(p,q,r) : p,r \in P \setminus \{q\} \text{ and $p$ and $r$ are distinct} \}.
\]
A starting observation is that, as long as $P$ is not contained in a single line, there exists $q \in P$ such that\footnote{Here and throughout the paper, the standard notation $\ll, \,\gg$ and respectively $O$ and $\Omega$ is applied to positive quantities in the usual way. That is, $X\gg Y$, $Y \ll X,$ $X=\Omega(Y)$ and $Y=O(X)$ all mean that $X\geq cY$, for some absolute constant $c>0$. }
\begin{equation} \label{linear}
|\mathcal A_q(P)| \gg |P|.
\end{equation}
  One method for proving \eqref{linear} is via an application of a classical result of Beck \cite{Be}, which states that any set of $n$ points in $\mathbb R^2$ either contains a line with $\Omega(n)$ points from $P$, or determines $\Omega(n^2)$ distinct lines. If there are $\Omega(|P|)$ points from $P$ on a line $\ell$, we may fix any point $q \in  P \setminus \ell$ and a point $p \in P \cap \ell$, and observe that the set
\[
\{\mathcal A(p,q,r) : r \in P \cap \ell, r \neq p \} \subseteq \mathcal A_q(P)
\]
has cardinality at least $\frac{|\ell \cap P|-1}{2} \gg |P|$. Otherwise, it follows from Beck's Theorem that there is a point $q \in P$ which determines $\Omega(|P|)$ directions with the other elements of $P$, and this immediately implies that $| \cA_q(P)| \gg |P|$. We will come to the problem of determining the correct constant hidden in the asymptotic notation soon.

This linear lower bound cannot be improved in general, and there are two simple families of examples which should be considered.
\begin{itemize}
    \item If $P$ is a set of $n$ points evenly distributed on a circle then $P$ determines $O(n)$ distinct angles. We can also adjoin the centre of the circle to the point set and still preserve the inequality $|\cA(P)| \ll n$.
    \item If we fix a point $p \in \mathbb R^2$ and then draw a line $\ell$ in the plane which does not contain $p$, we can then arrange $n$ points on $\ell$ so that the direction angles from $p$ to these points form an arithmetic progression. The set $P$ formed by taking the union of these $n$ points on $\ell$ with $p$ satisfies $|\cA(P)| \ll n$. Moreover, we can adjoin an additional point $p'$ to this set and preserve this inequality, where $p'$ is the reflection of $p$ in the line $ \ell$.
\end{itemize}

\begin{figure}[h]

\begin{tikzpicture}
    \def\n{10}
    \def\r{2}
    \draw (0,0) circle (\r);
    \node[circle,fill,inner sep=1pt] at (0,0) {};
    \foreach \s in {1,...,\n}{
        \pgfmathsetmacro\angle{360/\n * (\s - 1)}
        \coordinate (P\s) at (\angle:\r);
        \node[circle,fill,inner sep=1pt] at (P\s) {};
    }

\begin{scope}[shift={(0,5cm)}]
        \def\n{10}
    \def\r{2}
    \draw (0,0) circle (\r);
    \foreach \s in {1,...,\n}{
        \pgfmathsetmacro\angle{360/\n * (\s - 1)}
        \coordinate (P\s) at (\angle:\r);
        \node[circle,fill,inner sep=1pt] at (P\s) {};
    }
\end{scope}

  \begin{scope}[shift={(5cm,0)}, scale=0.5]  \draw (-1,0) -- (9,0);
    \node[circle,fill,inner sep=1pt] at (0,1) {};
    \node[circle,fill,inner sep=1pt] at (0,-1) {};
    \node[circle,fill,inner sep=1pt] at (0,0) {};
    \node[circle,fill,inner sep=1pt] at (0.3639,0) {};
    \node[circle,fill,inner sep=1pt] at (0.8390,0) {};
    \node[circle,fill,inner sep=1pt] at (1.7320,0) {};
      \node[circle,fill,inner sep=1pt] at (    5.6712,0) {};

      \end{scope}

  \begin{scope}[shift={(5cm,5cm)}, scale=0.5]  \draw (-1,0) -- (9,0);
    \node[circle,fill,inner sep=1pt] at (0,1) {};
    \node[circle,fill,inner sep=1pt] at (0,0) {};
    \node[circle,fill,inner sep=1pt] at (0.3639,0) {};
    \node[circle,fill,inner sep=1pt] at (0.8390,0) {};
    \node[circle,fill,inner sep=1pt] at (1.7320,0) {};
      \node[circle,fill,inner sep=1pt] at (    5.6712,0) {};

      \end{scope}

        \begin{scope}[shift={(2.5cm,8cm)}, scale=0.5]  \draw (-1,0) -- (9,0);
    \node[circle,fill,inner sep=1pt] at (0,0) {};
    \node[circle,fill,inner sep=1pt] at (7,0) {};
    \node[circle,fill,inner sep=1pt] at (4.5,0) {};
    \node[circle,fill,inner sep=1pt] at (4.9,0) {};
    \node[circle,fill,inner sep=1pt] at (0.7,0) {};
    \node[circle,fill,inner sep=1pt] at (1.5,0) {};
      \node[circle,fill,inner sep=1pt] at (    3.5,0) {};

      \end{scope}
    
\end{tikzpicture}

\caption{This picture shows five degenerate examples of sets of $n$ points determining $O(n)$ distinct angles (the first example determines only $2$ angles). The main result of this paper shows that, for $n$ sufficiently large, the estimate $|\cA(P)| \geq n^{1+c}$ holds for any $P$ which does not have one of these incidence structures.}

\end{figure}
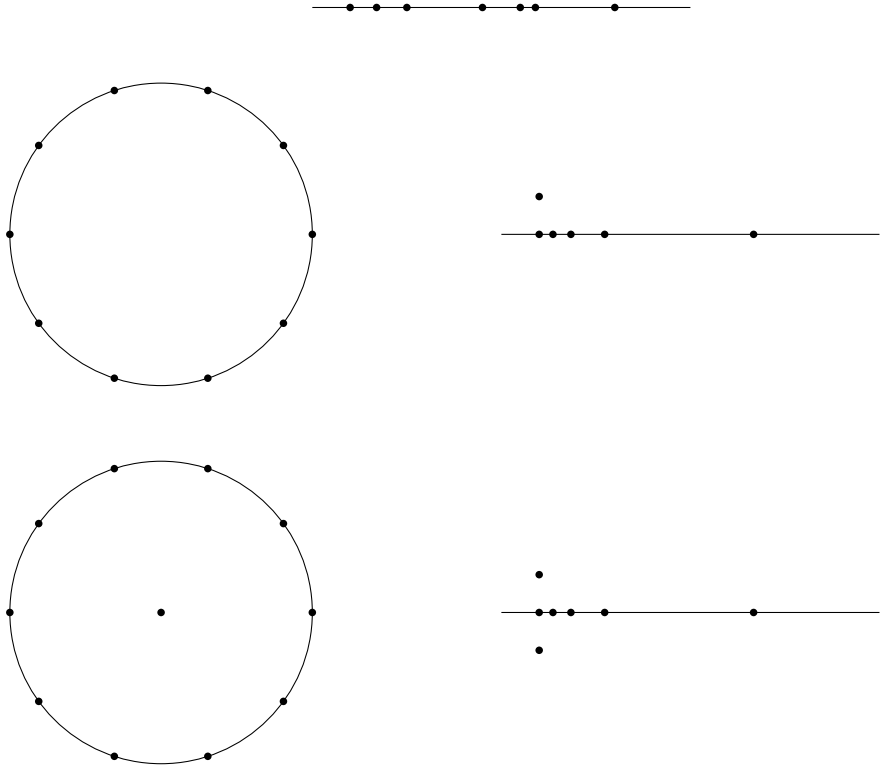

The main result of this paper gives a significantly improved lower bound for the size of $\cA(P)$ provided that $P$ does not contain a highly rich line or circle. Moreover, we obtain this lower bound for the pinned variant of the problem.

\begin{theorem}\label{thm:main}
   There exists an absolute constant $c>0$ such that, for $n$ sufficiently large, if $P \subset \mathbb R^2$ has cardinality $n$ then either there exists $q \in P$ such that $|\cA_q(P)| \geq n^{1+c}$, or $P$ is of one of the following forms: 
   \begin{enumerate}
       \item all of the points of $P$ are contained on a single line,
       \item $n-1$ points of $P$ are on a single line,
       \item $n-2$ points of $P$ are on a single line, and the other two points are symmetric with respect to the line,
       \item all of the points of $P$ are contained on a single circle,
       \item $n-1$ points of $P$ are on a circle and the other point is the centre of the circle.
   \end{enumerate}
\end{theorem}

Very little was previously known about the structure of point sets determining few angles. Even for the extreme case when $P$ is in general position - i.e. $P$ contains at most $2$ points on any line and at most $3$ points on any circle - the problem of obtaining a meaningful lower bound for $|\cA(P)|$ was wide open. Lower bounds on the size of $\cA(P)$ assuming additional information about $P$ were given in two recent papers; Konyagin, Passant and Rudnev \cite{KPR} shows that if $P$ is in convex position and is not contained on a circle then $|\cA(P)| \gg |P|^{5/4}$, and Roche-Newton \cite{Ro} showed that if $P=B \times B$ then $|\cA(P)| \gg |P|^{1+ \frac{1}{28}}$. Both \cite{KPR} and \cite{Ro} mention the open problem of determining a better than linear bound under the assumption that $P$ is a point set in general position. From the other side, Fleischmann et al. \cite{FKMPPW} showed that there exist point sets in general position determining $O(|P|^2)$ distinct angles, and a simpler explanation of this fact can also be found in \cite{KPR}. Ascoli et al. \cite{ABDLMMPRV} conjectured that any point set in general position determines $\Omega(|P|^2)$ distinct angles.

The problem of how many angles are determined by a point set may be seen as a variant of the more famous Erd\H{o}s distance problem. The celebrated work of Guth and Katz \cite{GuKa} essentially resolved this problem by proving that any point set $P \subset \mathbb R^2$ determines $\Omega\left ( n/ \log n \right )$ distinct distances. The integer grid $[\sqrt n] \times [\sqrt n]$ determines $\Theta \left ( n / \sqrt {\log n} \right) $ distinct distances. The inverse problem of classifying the sets which determine close to the minimum number of distances is wide open. It was suggested by Erd\H{o}s \cite{Er} that extremal configurations for this problem must have a lattice-like structure. In the same paper, it was conjectured that an optimal configuration of $n$ points must contain $\Omega(n^{1/2})$ points on a line, although it remains an open problem to prove that such a point set must contain $n^c$ points on a line for some absolute constant $c>0$. On the other hand, it was shown by Raz, Roche-Newton and Sharir \cite{RRS} that any set of $n$ points determining fewer than $n/5$ distances cannot contain more than $n^{43/52+o(1)}$ points on a line, improving an earlier result of Sheffer, Zahl and de Zeeuw \cite{SZZ}.

\subsection{A problem of Corr\'{a}di, Erd\H{o}s and Hajnal}

The first appearance that we have been able to find in the literature for the problem of determining the number of angles given by a point set was in a note of Erd\H{o}s \cite{ErdosConjecture}. He writes:

\textit{``Finally let me state an old and completely forgotten question of Corrádi, Hajnal and myself: Is it true that if there are given n points in the plane, not all on a line, then they determine at least $n-2$ different angles? (0 and $\pi$ are counted as different but angles greater than $\pi$ are
not allowed)."}

The problem also appeared in a paper of Erd\H{o}s and Purdy \cite[p. 846]{ErdosPurdy}. The example of the regular $n$-gon shows that a positive answer to this question would be a sharp result. We use \Cref{thm:main} to give a positive answer to this question for all $n$ sufficiently large. In fact we even further prove that the \textit{only} configuration achieving $n-2$ angles is the regular $n$-gon.

\begin{theorem}\label{thm:classification}
    For $n$ sufficiently large, every non-collinear point set $P \subseteq \mathbb R^2$ of size $n$ determines at least $n-2$ distinct angles. Furthermore, the only point configuration of $n$ non-collinear points which achieves exactly $n-2$ distinct angles is the regular $n$-gon.
\end{theorem}

\subsection{A sketch of the proof of \Cref{thm:main}}

Before beginning, we give a rough sketch of the proof, focusing mainly on the case when the point set is in general position, as most of the key ideas already occur here. Given a point set $P$ in general position, consider any fixed point $q \in P$, and suppose that $\mathcal A_q(P)$ is small. If we consider the set $P$ as complex numbers, the angles defined between a point of $P$, the point $q$, and the horizontal, can be encoded by the ratios $\frac{p-q}{\overline{p} - \overline{q}}$, for $p \in P$. Let us call the set of these ratios $D_q$. It then follows that the ratio set $D_q D_q^{-1}$ encodes the angles of $\mathcal A_q(P)$. Therefore this ratio set is small. We then apply the weak polynomial Freiman-Ruzsa theorem to this set, concluding that a large chunk of $D_q$ is contained in a low rank subgroup of $\mathbb C^*$.

 We repeat this with three points $p_1,p_2,p_3 \in P$, and suppose for a contradiction that all three of the points determine few pinned angles. After some technical work and making three applications of the Freiman-Ruzsa Theorem, we can place a large chunk of all three set $D_{p_1}, D_{p_2},$ and $D_{p_3}$ within the same low-rank subgroup $\Gamma$. On the other hand, each point $p \in P \setminus \{p_1,p_2,p_3\}$ yields a triple of ratios - one from each of the sets $D_{p_i}$. This gives us many solutions to a system of three equations involving elements of $\Gamma$, which can be reduced to a single linear equation with variables in $\Gamma$. We are now in a situation whereby the subspace theorem is applicable. We need to take care of possibly degenerate solutions to this equation, but it turns out that the general position assumption can be used to handle these cases. We arrive at a contradiction - the linear equation has many solutions, contradicting the subspace theorem - and so at least one of the three fixed pins $p_1,p_2,p_3$ must determine many angles as a centre with the remaining points of $P$.
 
 In order to obtain the more general form of Theorem \ref{thm:main}, we work with the weaker assumption that no $T$ points of $P$ lie on a line or circle for some $T \leq n^{9/10}$. The proof is broadly the same, although we need to take more care in choosing the three initial points so that they satisfy an incidence structure which allows the technical work in the Freiman-Ruzsa applications to not be too wasteful. This is handled via pigeonholing and an application of Beck's Theorem. The factor of $T$ causes a quantitative loss when it comes to the subspace theorem application, but this can be controlled as long as $T$ is significantly smaller than $n$. Finally, for the case when $P$ contains super rich lines or circles (i.e. when $T > n^{9/10}$), we can use a sum-product type result of Elekes, Nathanson and Ruzsa \cite{ENR} to prove that $P$ determines many pinned angles. This step crucially requires a small number of points off the rich line or circle. 

\section{Some pre-processing via Beck's Theorem} \label{sec:beck}

In this section we apply some pre-processing to our point set, which is designed to find three distinct points $p_1,p_2,p_3 \in P$ and a large subset $P' \subseteq P$  such that all of the lines defined between one of $p_1,p_2,p_3$ and the set $P'$ have few points of $P$ on them. To do this we use a variation of Beck's Theorem. We need slightly more information than is given by the classical formulation of the statement, and the following version of Beck's Theorem can be derived from Theorem 6 in Solymosi, Tardos and T\'{o}th \cite{STT}.
\begin{theorem}\label{thm:becks}
    Let $P$ be a set of $n$ points in $\mathbb R^2$. There exist positive absolute constants $\alpha$, $\beta$, and $\gamma$ such that at least one of the following is true.
  \begin{enumerate}
      \item There is a line containing at least $\alpha n$ points of $P$.
      \item There exists a set of at least $\beta n^2$ pairs of distinct points of $P$, such that each pair defines a line containing at most $\gamma$ points of $P$.
  \end{enumerate}
\end{theorem}
We will prove the following lemma.
\begin{lemma}\label{lem:preprocessing} Let $\alpha,\beta$ and $\gamma$ be the constants obtained from \Cref{thm:becks}. Suppose that $n$ is sufficiently large and let $P$ be a set of $n$ points in the plane with no line containing $\alpha n$ points of $P$. Then there exist three distinct points $p_1,p_2,p_3 \in P$, and a subset $Q \subseteq P \setminus \{p_1,p_2,p_3\}$ of size $|Q| \geq 4 \beta^3 n$ such that for any $i=1,2,3$, any line through the point $p_i$ contains at most $\gamma$ points of $P$.
\end{lemma}

\begin{proof}

Since no line contains $\alpha n$ points of $P$, \Cref{thm:becks} implies that there is a set $\mathcal{H} \subseteq P^2$ of distinct pairs, such that  $|\mathcal{H}| \geq \beta n^2$, and each pair from $\mathcal{H}$ defines a line containing at most $\gamma$ points of $P$. We now define a graph $G = (V,E)$ in the following way - the vertex set $V$ is the point set $P$, and two vertices $p_1$ and $p_2$ form an edge if $(p_1,p_2) \in \mathcal{H}$. Applying H\"{o}lder's inequality, we have
\begin{align*}
2\beta n^2& \leq 2|E| = \sum_{p \in P}|\{p_1  \in P: (p,p_1) \in E\}|\\ & \implies 8\beta^3 n^6 \leq n^2 \sum_{p \in P} |\{(p_1,p_2,p_3) \in P^3 : (p,p_1),(p,p_2),(p,p_3) \in E\}|  \\
& \ \quad \qquad \qquad=n^2 \sum_{(p_1,p_2,p_3) \in P} |\{p \in P : (p,p_1),(p,p_2),(p,p_3) \in E\}|
\end{align*}
Taking $n$ sufficiently large ensures that at least half of the contributions to this latter sum come from distinct triples.
Therefore, pigeonholing implies that there exists a triple of distinct points $(p_1,p_2,p_3) \in P^3$ such that
$$|\{p \in P : (p,p_1),(p,p_2),(p,p_3) \in E\}| \geq 4\beta^3 n.$$
Taking $Q$ to be this set of points finishes the proof of \Cref{lem:preprocessing}.

\end{proof}

\section{Finding a good subgroup via the weak polynomial Freiman-Ruzsa theorem} \label{sec:freiman}

From here on we will typically consider our point set $P \subseteq \mathbb R^2$ as a set of complex numbers via $\mathbb R^2 \cong \mathbb C$. We define, for any $p \in \mathbb R^2 \cong \mathbb C$ and any $Q \subset \mathbb C \setminus \{p\}$,
\[
D_p(Q):= \left\{ \frac{z-p}{\overline{z}-\overline{p}} : z \in Q \right\} \subset \mathbb C.
\]
Note that for each $z \in \mathbb C \setminus \{p\}$, the corresponding $\frac{z-p}{\overline{z} - \overline{p}}$ is a complex number of absolute value one, whose argument is twice the oriented angle that the line segment from $p$ to $z$ makes with the positive horizontal axis. 
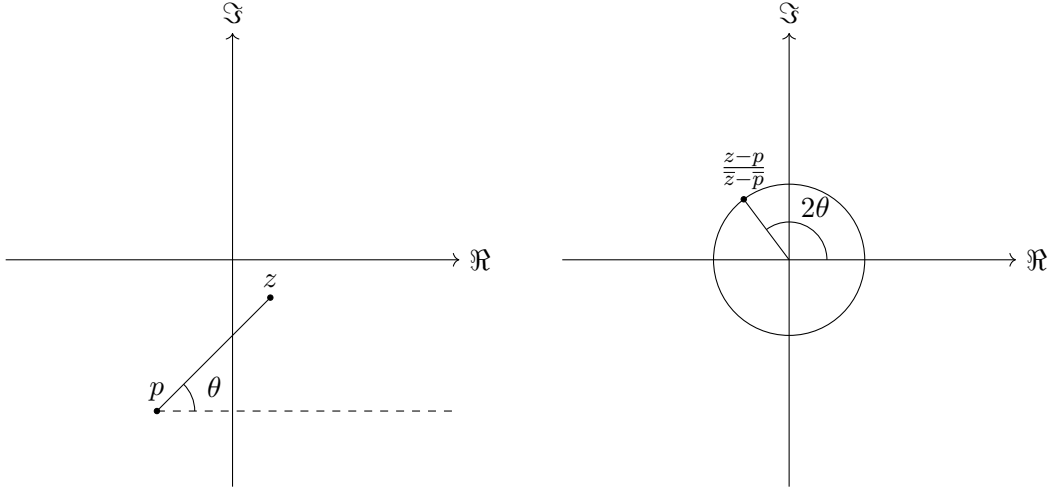
\begin{figure}[h!]
\begin{center}
\begin{minipage}{0.4 \linewidth}
    \centering
    \begin{tikzpicture}[scale=1]
\coordinate (x) at (3,1.7);
\coordinate (O) at (0,0);
\coordinate (p) at (-1,-2);
\coordinate (z) at (0.5,-0.5);
\coordinate (line) at (3,-2);
    \draw[->] (-3,0) -- (3,0) node[right] {$\Re$};
    \draw[dashed] (p) -- (3,-2);
    \draw[->] (0,-3) -- (0,3) node[above] {$\Im$};
        \draw (z) -- (p);
    \filldraw[black] (p) circle (1pt) node[above] {$p$};
    \filldraw[black] (z) circle (1pt) node[above] {$z$};
    \pic[
    draw,
    "$\theta$",
    angle radius=0.5cm,
    angle eccentricity=1.65
] {angle=line--p--z};
\end{tikzpicture} 
\end{minipage} \hspace{5mm}
\begin{minipage}{0.4 \linewidth}
    \centering
    \begin{tikzpicture}[scale=1]
\coordinate (x) at (2,0);
\coordinate (O) at (0,0);
\coordinate (p) at (-0.6,0.8);
\coordinate (p') at (1.5,1.7);
    \draw (0,0) circle (1);
    \draw[->] (-3,0) -- (3,0) node[right] {$\Re$};
    \draw[->] (0,-3) -- (0,3) node[above] {$\Im$};
        \draw (0,0) -- (p);
    \filldraw[black] (p) circle (1pt) node[above] {$\frac{z-p}{\overline{z} - \overline{p}}$};
    \pic[
    draw,
    "$2\theta$",
    angle radius=0.5cm,
    angle eccentricity=1.55
] {angle=x--O--p};
\end{tikzpicture}
\end{minipage}
    \caption{The angle that $p$ makes with $z$ and the horizontal is half the argument of {$\frac{z-p}{\overline{z} - \overline{p}}$}}
    \label{fig:placeholder}
\end{center}
\end{figure}

\begin{lemma} Let $p \in \mathbb R^2 \cong \mathbb C$ and let $Q \subseteq \mathbb C \setminus \{p\}$. Suppose that $z_1$ and $z_2$ are distinct elements of $Q$ such that $(z_1,p,z_2)$ is  not a collinear triple and let $\phi_1,\phi_2$ be the internal and external angles defined by the triple $(z_1,p,z_2)$. Let $x = \frac{z_1-p}{\overline{z_1} - \overline{p}}$ and $y = \frac{z_2-p}{\overline{z_2} - \overline{p}}$ be the two corresponding elements of $D_p(Q)$. Then we have $\{ x/y, y/x \} = \{e^{i2\phi_1}, e^{i2\phi_2}\}$ - that is, these two ratios encode the internal and external angles. In particular,
\begin{equation*} \label{dirangles}
|\cA_p(Q)| \geq \frac{|D_p(Q) / D_p(Q)| -1}{2}.
\end{equation*}
\end{lemma}

\begin{proof}
Consider the two points $z_1,z_2 \in Q$. Suppose that these points make angles $\theta_1, \theta_2$ with the horizontal. Suppose WLOG that $\theta_2 \geq \theta_1$, in which case either the internal or external angle of the triple $(z_1,p,z_2)$ is given by $\theta_2 - \theta_1$. Note that both complex numbers $x = e^{i2\theta_1}$ and $y = e^{i2\theta_2}$ appear within $D_p(Q)$. Therefore within the ratio set $D_{p}(P) D_{p}(P)^{-1}$ we find the complex number $y/x = e^{i2(\theta_2 - \theta_1)}$, which encodes this (either internal or external) angle. The complex conjugate of this number, namely $e^{i2(\theta_1 - \theta_2)}$, also appear in the ratio set. Now, it is clear that if $\theta_2 - \theta_1$ is the internal (resp. external) angle, then $2\pi + \theta_1 - \theta_2$ is the external (resp. internal) angle, simply because they sum to $2\pi$. However we then see that since $x/y = e^{i2(\theta_1 - \theta_2)} = e^{i2(2\pi + \theta_1 - \theta_2)}$, we have that two times this second angle also appears in the ratio set.  This proves the first part of the lemma.

Let $\cA_p(Q)^* = \cA_p(Q) \setminus \{0, \pi \}$ (i.e. those angles pinned at $p$ arising from non-collinear triples) and let $\overline{\cA_p(Q)^*}$  denote the set of exterior angles determined by non-collinear triples. Note also that $|\cA_p(Q)^*|=|\overline{\cA_p(Q)^*}|$. The first part of the statement of the lemma then implies that 
\[
|D_p(Q)/D_p(Q)|=|\cA_p(Q)^*| + |\overline{\cA_p(Q)^*}| +1=2|\cA_p(Q)^*|+1 \leq 2|\cA_p(Q)|+1,
\]
which completes the proof.

\end{proof}


With this in mind, if we have a set $P$ determining few distinct angles centred at each point $p \in P$, we must be in the situation where each of the sets $D_p(P)$ has multiplicative structure in $\mathbb C$. To make use of this information, we will apply some of the latest technology with regard to the Freiman-Ruzsa Theorems, giving precise quantitative and qualitative information about sets which determine few products. The following result of Harrison, Mudgal and Schmidt \cite{HMS} will be used.


\begin{theorem}[\cite{HMS}, Lemma 6.2] \label{thm:weakPFR}
    Let $A \subset \mathbb C^*$ such that $|AA| \leq K|A|$. Then there exists a multiplicative subgroup $\Gamma$, an absolute constant $C>0$, and a subset $A' \subset A$ such that
    \[
    |A'| \gg |A|/K^C, \,\,\,\, \rank (\Gamma) \ll \log K, \,\,\,\,\, A' \subset \Gamma.
    \]
\end{theorem}

We note that \Cref{thm:weakPFR} uses the breakthrough results of Gowers, Green, Manners and Tao \cite{GGMT} in resolving the weak polynomial Freiman-Ruzsa conjecture over $\mathbb Z$. This quantitative breakthrough is crucial to our analysis in this paper. If we were to run through the argument of this paper using the previous best known Freiman-Ruzsa type bounds, we would instead recover the weaker lower bound $|\cA_p(P)| \geq n \exp ( \log^c n)$ in the statement of \Cref{thm:main}.

The main result of this section is the following.

\begin{lemma} \label{lem:subspacetrap}
Let $n \in \mathbb N$ be sufficiently large and let $T \leq \alpha n$, where $\alpha$ is the absolute constant from the statement of \Cref{thm:becks}. Let $P \subset \mathbb R^2$ be a set with cardinality $n$ with no line containing more than $T$ points and such that $|\cA_p(P)| \leq Ln$ for all $p \in P$. Then there is an absolute constant $C'$, three distinct points $a_1,a_2,a_3 \in P$, a multiplicative subgroup $\Gamma \leq \mathbb C^*$ with $\rank(\Gamma) \ll \log L$, and a subset $P' \subset P \setminus \{a_1,a_2,a_3\}$ such that $|P'| \gg n/L^{C'}$ and
\begin{equation} \label{containment}
D_{a_i}(P') \subseteq \Gamma , \,\,\,\, \forall \,\, 1 \leq i \leq 3.
\end{equation}

\end{lemma}


\begin{proof}

Apply \Cref{lem:preprocessing} to obtain $a_1,a_2,a_3$ and $Q \subset P$ such that $|Q| \gg n$ and for every $1 \leq i \leq 3$ and every $q \in Q$, the line connecting $a_i$ and $q$ contains at most $\gamma$ points, where $\gamma$ is an absolute constant. Write $D_{a_i}$ as a shorthand for $D_{a_i}(Q)$. 

Furthermore, since each line through $a_i$ can only contain $\gamma$ points of $Q$, we have that $$|D_{a_i}| \gg n.$$
Therefore by our assumption on the number of angles defined from any point of $P$, we have that for $1 \leq i \leq 3$, the set $D_{a_i}(P)$ satisfies
\[
|D_{a_i}(P)D_{a_i}(P)^{-1}| \ll  |\cA_{a_i}(P)| \leq Ln.
\]
By the Ruzsa Triangle inequality 
\[|D_{a_1}D_{a_1}| \leq |D_{a_1}(P)D_{a_1}(P)|  \leq L^2n  \ll L^2 |D_{a_1}|.
\]
Apply \Cref{thm:weakPFR} to get a large subset $E_1 \subseteq D_{a_1}$ contained in a low rank subgroup $\Gamma_1$. Now choose a point set $P_1 \subseteq Q$ with one point representing each direction in $E_1$. That is, we identify a set $P_1$ such that
\[
E_1=D_{a_1}(P_1),
\]
by simply choosing, for each element $d \in E_1$, one point from $Q$ on the line through $a_1$ corresponding to the direction $d$. We have
\[
|P_1|=|E_1| \gg |D_{a_1}| / L^{2C} \gg n / L^{2C}, \,\,\, \rank(\Gamma_1)\ll  \log L,
\]
and $E_1 \subseteq \Gamma_1$. Here and throughout the proof of this lemma, $C$ is the absolute constant from the statement of \Cref{thm:weakPFR}. In particular, we have
\begin{equation} \label{contain1}
D_{a_1}(P_1) \subseteq \Gamma_1.
\end{equation}
Next, consider the directions from $a_2$ to $P_1$, i.e. the set
\[
D_{a_2}(P_1)\subseteq D_{a_2}(P).
\]
We have
\[
|D_{a_2}(P_1)D_{a_2}(P_1) | \leq |D_{a_2}(P)D_{a_2}(P)| \leq L^2n \ll L^{2+2C}|P_1|  \ll L^{2+2C}|D_{a_2}(P_1)|.
\]
The last inequality in the line above uses the fact that since each line through $a_2$ can contain at most $\gamma$ points of $Q$ (and $ P_1 \subseteq Q$), we have $$|D_{a_2}(P_1)| \gg |P_1|.$$
A second application of \Cref{thm:weakPFR} gives a set $E_2 \subseteq D_{a_2}(P_1)$ and a subgroup $\Gamma_2$ such that
\[
    |E_2| \gg \frac{|D_{a_2}(P_1)|}{L^{C(2+2C)}} \geq \frac{|D_{a_2}(P_1)|}{L^{3C^2}}, \,\,\,\, \rank (\Gamma_2) \ll \log L, \,\,\,\,\, E_2 \subseteq \Gamma_2.
    \]
Let $P_2 \subseteq P_1$ be a set of elements of $P_1$ corresponding to the direction set $E_2$. That is, $P_2$ is a subset of $P_1$ satisfying
\[
E_2= D_{a_2}(P_2).
\]
In particular,
\[
|P_2|=|E_2| \gg \frac{|D_{a_2}(P_1)|}{L^{3C^2}} \gg \frac{ |P_1|}{L^{3C^2}} \gg \frac{n}{ L^{4C^2}} 
\]
and
\begin{equation} \label{contain2}
D_{a_2}(P_2) \subseteq \Gamma_2.
\end{equation}
Repeating this process a third and final time, consider the directions from $a_3$ to $P_2$, i.e. the set    
\[
D_{a_3}(P_2)\subseteq D_{a_3}(P).
\]
We have
\[
|D_{a_3}(P_2)D_{a_3}(P_2)| \leq |D_{a_3}(P)D_{a_3}(P)| \leq L^2n \ll L^{4C^2+2}|P_2| \ll L^{4C^2+2}|D_{a_3}(P_2)| \leq L^{5C^2}|D_{a_3}(P_2)|.
\]
Here we have used the fact that
$$|D_{a_3}(P_2)| \gg |P_2|,$$
which again derives from the relationship between $a_3$ and $Q$ which is given by \Cref{lem:preprocessing}. A third application of \Cref{thm:weakPFR} gives a set $E_3 \subseteq D_{a_3}(P_2)$ and a subgroup $\Gamma_3$ such that
\begin{gather*}
    |E_3| \gg \frac{ |D_{a_3}(P_2)|}{L^{5C^3}} \gg \frac{ |P_2|}{L^{5C^3}} \gg \frac{ n}{L^{9C^3}}, 
    \\ \rank (\Gamma_3) \ll \log L,
    \end{gather*}
    and $E_3 \subseteq \Gamma_3$. Let $P_3 \subseteq P_2 \subseteq P_1$ be a representative point set corresponding to the direction set $E_3$, i.e. a set $P_3$ satisfying    $D_{a_3}(P_3)=E_3$. In particular, we have
\begin{equation} \label{contain3}
D_{a_3}(P_3) \subseteq \Gamma_3.
\end{equation}
and
\[
|P_3| \gg |E_3| \gg \frac{n}{L^{9C^3}}.
\]
     This set $P_3$ plays the role of the set $P'$ in the statement of the lemma. The claim on the size of $P'$ is satisfied by setting $C'=9C^3$.

    Multiplying the three subgroups together, we obtain a group $\Gamma :=\Gamma_1 \Gamma_2 \Gamma_3$, such that $\Gamma$ is of rank $O\left (\log L \right )$. This is the subgroup from the statement of the lemma. It remains to verify the claim \eqref{containment}. Indeed, for $1 \leq i \leq 3$, we have
    \[
    D_{a_i}(P') \subseteq D_{a_i}(P_i) \subseteq \Gamma_i \subseteq \Gamma,
    \]
    where the second containment above follows from \eqref{contain1}, \eqref{contain2} and \eqref{contain3}.
\end{proof}
\section{Applying the subspace theorem}
In this section we will apply the subspace theorem to prove the following lemma.

\begin{lemma} \label{lem:subspaceapp}
    Let $Q \subset \mathbb R^2 \cong  \mathbb C$ and let $a_1,a_2$ and $a_3$ be three distinct points of $\mathbb R^2 \setminus Q$ such that no more than $T$ points from the set $Q \cup \{a_1,a_2,a_3\}$ lie on a line or circle. Suppose that there is a rank $r$ subgroup $\Gamma \leq \mathbb C^*$ such that
    \[
    D_{a_i}(Q) \subseteq \Gamma
    \]
    holds for $1 \leq i \leq 3$. Then $|Q| =O( Te^{C_1r})$, where $C_1$ is an absolute constant.
\end{lemma} 

We will use the following version of the subspace theorem, due to Evertse, Schlickewei, and Schmidt \cite{ESS}.

\begin{theorem}[Evertse, Schlickewei, Schmidt]\label{thm:subspace}
     For all $a_1,...,a_k \in \mathbb C^*$, the number of non-degenerate solutions to the linear equation
     $$a_1 x_1 + ... + a_k x_k =1$$
     with each $x_i$ within a rank $r$ multiplicative subgroup $\Gamma$ is bounded by $O\left( \exp{ \left( (6k)^{3k}(r+1) \right) }\right)$. Here, a non-degenerate solution is a solution $(x_1,...,x_k) \in \Gamma^k$ where no sub-sum in the linear equation is equal to zero.
 \end{theorem}

 Let us now prove \Cref{lem:subspaceapp}.

 \begin{proof}For each $z \in Q$, we define the three numbers
     $$s =  \frac{z - a_1}{\overline{z} - \overline{a_1}}, \quad t =  \frac{z - a_2}{\overline{z} - \overline{a_2}}, \quad  u = \frac{z - a_3}{\overline{z} - \overline{a_3}}.$$
     We now define a map $\phi$ from $Q$ to $\Gamma^5$ in the following way:
     $$\phi(z) = \left( s,t, \frac{s}{u}, \frac{t}{u}, \frac{st}{u} \right).$$
     This map does indeed go to $\Gamma^5$ since each of $s,t,u \in \Gamma$.
     We will begin with the following small lemma.

     \begin{lemma}\label{lem:variablesdefinez}
         Fixing any one value of $s,t,u$, or ratio of two of these variables, there are at most $T$ values of $z \in Q$ which give this fixed value.
     \end{lemma}

     \begin{proof}[Proof of \Cref{lem:variablesdefinez}]
         Suppose we fix a single value of $s =  \frac{z - a_1}{\overline{z} - \overline{a_1}}$. The points of $Q$ which yield this fixed value of $s$ correspond to points of $Q$ lying on a fixed line through $a_1$. Since there are at most $T$ points of $Q$ collinear, at most $T$ values of $z \in Q$ can give this fixed value. Precisely the same argument applies to fixed values of $t$ and $u$.

         Now suppose we fix the ratio $\frac{s}{t} = \frac{(z-a_1)(\overline{z} - \overline{a_2})}{(\overline{z} - \overline{a_1})(z-a_2)}$. Assume that there are $T+1$ values $z_1,z_2,...,z_{T+1} \in Q$ which give this fixed ratio value. Consider the pair $z_1,z_i$ for $i=2,...,T+1$. Since they must give the same value of $\frac{s}{t}$, we have
         $$ \frac{(z_1-a_1)(\overline{z_1} - \overline{a_2})}{(\overline{z_1} - \overline{a_1})(z_1-a_2)} =  \frac{(z_i-a_1)(\overline{z_i} - \overline{a_2})}{(\overline{z_i} - \overline{a_1})(z_i-a_2)} \implies \frac{(z_1 - a_1)(z_i - a_2)}{(z_1 - a_2) (z_i-a_1)} = \frac{(\overline{z_1} - \overline{a_1})(\overline{z_i} - \overline{a_2})}{(\overline{z_1} - \overline{a_2}) (\overline{z_i}-\overline{a_1})}.$$

         The left hand side of this last equality is actually the cross ratio of the four complex numbers $z_1,z_i,a_1,a_2$. Furthermore, the right hand side is precisely its complex conjugate - since these are equal, we conclude that this cross ratio is real. We then need the following result concerning cross ratios, see for instance \cite[Section 3.2, Theorem 13]{Ahlfors}

          \begin{lemma}\label{lem:crossratio}
     The cross ratio of four complex numbers $a,b,c,d$ is real if and only if the four points are collinear or co-circular.
 \end{lemma}
        Applying \Cref{lem:crossratio}, we conclude that $a_1,a_2,z_1,$ and $z_i$ are either collinear or co-circular for each $i=2,...,T+1$. Depending on whether $z_1,a_1,a_2$ define a line or a circle, we find either a line or a circle containing all of $z_1,z_2,...,z_{T+1}$ - in either case contradicting our assumption that at most $T$ points of $Q$ are collinear or co-circular. Therefore at most $T$ points $z \in Q$ can give this fixed ratio $\frac{s}{t}$. Precisely the same argument can be applied to the other ratios $\frac{s}{u}$ and $\frac{t}{u}$. This concludes the proof of \Cref{lem:variablesdefinez}.
     \end{proof}
     We now return to proving \Cref{lem:subspaceapp}. Consider the equations
     $$s =  \frac{z - a_1}{\overline{z} - \overline{a_1}}, \quad t =  \frac{z - a_2}{\overline{z} - \overline{a_2}}, \quad  u = \frac{z - a_3}{\overline{z} - \overline{a_3}}.$$
     Using the first and second equations to eliminate $z$ yields the single equation
     $$s(a_3-a_2) + t(a_1-a_3) + u(a_2 - a_1) + st(\overline{a_2} - \overline{a_1}) +  su(\overline{a_1} - \overline{a_3}) + tu(\overline{a_3} - \overline{a_2})=0.$$
     For any $z \in Q$ we have $s,t,u \neq 0$. We can then divide this equation by $u(a_1-a_2)$ and rearrange, obtaining the equation
     \begin{equation} \label{eqn:subspaceeqnstu}
     s \left( \frac{\overline{a_3} - \overline{a_1}}{a_2 - a_1} \right) + t \left( \frac{\overline{a_2} - \overline{a_3}}{a_2 - a_1} \right) + \left( \frac{s}{u} \right) \left( \frac{a_2 - a_3}{a_2 - a_1} \right) + \left( \frac{t}{u} \right) \left( \frac{a_3 - a_1}{a_2 - a_1} \right) + \left( \frac{st}{u} \right) \left( \frac{\overline{a_1} - \overline{a_2}}{a_2 - a_1} \right) =1.\end{equation}
     We conclude that for any $z \in Q$, the image $\phi(z)$ gives a solution in $\Gamma^5$ to the linear equation
     \begin{equation} \label{eqn:subspaceeqnxvars}
     x_1 \left( \frac{\overline{a_3} - \overline{a_1}}{a_2 - a_1} \right) + x_2 \left( \frac{\overline{a_2} - \overline{a_3}}{a_2 - a_1} \right) + x_3 \left( \frac{a_2 - a_3}{a_2 - a_1} \right) + x_4 \left( \frac{a_3 - a_1}{a_2 - a_1} \right) + x_5 \left( \frac{\overline{a_1} - \overline{a_2}}{a_2 - a_1} \right) =1.\end{equation}
 Applying \Cref{thm:subspace} to this equation, we conclude that there are at most $O\left( \exp{ \left( (30)^{15}(r+1) \right) }\right)$ non-degenerate solutions $(x_1,x_2,x_3,x_4,x_5) \in \Gamma^5$ to this equation. Each such non-degenerate solution can be mapped to at most $T$ times by a point $z \in Q$, by \Cref{lem:variablesdefinez}. Therefore the number of $z \in Q$ which map to a non-degenerate solution of \cref{eqn:subspaceeqnxvars} under $\phi$ is at most $O\left(T \exp{ \left( (30)^{15}(r+1) \right) }\right) \ll T e^{C_2r}$.

 \subsection{Bounding degenerate solutions.}
 
 We now need to bound the number of points $z\in Q$ which map to a degenerate solution of \cref{eqn:subspaceeqnxvars}. For each degenerate solution $(x_1,x_2,x_3,x_4,x_5)\in \Gamma^5$, there is a (possibly not unique) maximal sub-sum, call it $\Sigma$, which vanishes - if this is not unique, choose a maximal vanishing sub-sum arbitrarily. Here `maximal' means with respect to the number of terms of \cref{eqn:subspaceeqnxvars}. For any such $\Sigma$, we let $\Sigma'$ be the complement of the terms of $\Sigma$. This means that any degenerate solution $(x_1,x_2,x_3,x_4,x_5) \in \Gamma^5$ to \cref{eqn:subspaceeqnxvars} has been uniquely associated with some maximal vanishing sub-sum $\Sigma$, and furthermore the complement $\Sigma'$ satisfies $\Sigma' =1$. Moreover, the associated solution to $\Sigma'=1$ cannot be degenerate, as if there were some vanishing sub-sum $\Sigma'' = 0$ within $\Sigma'$, we would have that $\Sigma + \Sigma'' = 0$ is a vanishing sub-sum of the degenerate solution, contradicting maximality of $\Sigma$. 
 
 Suppose that $\Sigma'$ includes any of the first four terms of \cref{eqn:subspaceeqnxvars}. We can then apply \Cref{thm:subspace} to the equation $\Sigma'=1$, obtaining at most $e^{C_3r}$ non-degenerate solutions to the equation $\Sigma' =1$. Since this equation includes one of the first four variables, \Cref{lem:variablesdefinez} can be applied to show that such solutions can only be mapped to by at most $T$ values of $z\in Q$. Therefore in each of the $O(1)$ cases where the maximal $\Sigma$ yields a sum $\Sigma' = 1$ with $\Sigma'$ involving one of the first four terms, there are at most $Te^{C_3r}$ values of $z$ which can give such a degenerate solution. 

 The only case left over is when $\Sigma$ includes all of the first four terms - specifically, we have
 $$s \left( \frac{\overline{a_3} - \overline{a_1}}{a_2 - a_1} \right) + t \left( \frac{\overline{a_2} - \overline{a_3}}{a_2 - a_1} \right) + \left( \frac{s}{u} \right) \left( \frac{a_2 - a_3}{a_2 - a_1} \right) + \left( \frac{t}{u} \right) \left( \frac{a_3 - a_1}{a_2 - a_1} \right)=0,$$
 and thus
 \[
 s(\overline{a_3} - \overline{a_1}) + t(\overline{a_2} - \overline{a_3}) + \left( \frac{s}{u} \right) \left( a_2 - a_3 \right) + \left( \frac{t}{u} \right) \left( a_3 - a_1 \right)=0.
 \]
 After dividing through by $t(\overline{a_3} - \overline{a_2})$ (which is never zero) and rearranging, this becomes the equation
 \begin{equation} \label{eqn:smallersubspace}
     \left(\frac{s}{t}\right) \left( \frac{\overline{a_3} - \overline{a_1}}{\overline{a_3} - \overline{a_2}} \right) + \left( \frac{1}{u} \right) \left( \frac{a_3 - a_1}{\overline{a_3} - \overline{a_2}}\right) + \left( \frac{s}{ut} \right) \left( \frac{a_2 - a_3}{\overline{a_3} - \overline{a_2}} \right)=1
 \end{equation}
 We can now apply \Cref{thm:subspace} to this equation, again giving at most $T e^{C_4 r}$ values of $z \in Q$ which can map to a non-degenerate solution within $\Gamma^3$ of \cref{eqn:smallersubspace}, since such a solution determines the value of $u$, and we can then apply \Cref{lem:variablesdefinez}. We must now consider degenerate solutions to \cref{eqn:smallersubspace}. Any degenerate solution where a single term in the sum vanishes cannot be mapped to by any $z \in Q$, since all of $\frac{s}{t}$, $\frac{1}{u}$, and $\frac{s}{ut}$ are non-zero. We are then left only with solutions where a sum of two terms vanish - implying that the third remaining term is equal to one. If this remaining term is the first or second term of \cref{eqn:smallersubspace}, then the value of either $\frac{s}{t}$ or of $u$ is determined, and then applying \Cref{lem:variablesdefinez} implies that there can be only $T$ values of $z$ which map to such solutions. The final case is therefore when we have the equations
 $$ \left(\frac{s}{t}\right) \left( \frac{\overline{a_3} - \overline{a_1}}{\overline{a_3} - \overline{a_2}} \right) + \left( \frac{1}{u} \right) \left( \frac{a_3 - a_1}{\overline{a_3} - \overline{a_2}}\right) = 0, \quad  \left( \frac{s}{ut} \right) \left( \frac{a_2 - a_3}{\overline{a_3} - \overline{a_2}} \right)=1.$$
 The second equation can be rearranged to the form $\frac{s}{t} = u \left( \frac{\overline{a_3} - \overline{a_2}}{a_2 - a_3} \right)$, which upon substitution into the first equation yields a quadratic equation for $u$. Therefore there are at most two values of $u$ possible for this case, and so at most $2T$ values of $z \in Q$ can map to a solution of this form.

 Putting all these cases together, we conclude that the total number of possible values of $z$ is $O(Te^{C_1r})$ for some absolute constant $C_1$. This concludes the proof of \Cref{lem:subspaceapp}.
 \end{proof}

 \section{The case of no rich lines or circles }

A key step towards proving the main result of this paper is the following theorem. In particular, this result proves Theorem \ref{thm:main} unless there is a line or a circle containing $n^{9/10}$ points of $P$. The proof is a simple combination of the results we have established in the previous two sections.

\begin{theorem}\label{thm:mainnearly}
    For all $\epsilon>0$ there exists a constant $c>0$ and $n_0\in \mathbb N$ such that if $P \subset \mathbb R^2$ has cardinality $n \geq n_0$ and $|\cA_p(P)| \leq n^{1+c}$ for all $p \in P$, then there is a line or circle containing at least $n^{1-\epsilon}$ elements of $P$.
\end{theorem}

\begin{proof}
Fix $\epsilon >0$ and let $P\subseteq \mathbb R^2$ satisfy that $|\cA_p(P)| \leq n^{1+c}$ for all $p \in P$. Let $T$ denote the maximum number of points from $P$ on a line or circle. We will show that, by taking $c>0$ sufficiently small, we can ensure that $T \geq n^{1-\epsilon}$. 

We can assume $T \leq \alpha n$ where $\alpha$ is the constant from the statement of  \Cref{lem:subspacetrap}, as otherwise we are done. An application of \Cref{lem:subspacetrap} gives us three elements $a_1,a_2,a_3 \in P$, a subset $P' \subset P \setminus \{ a_1,a_2,a_3 \}$ and a subgroup $\Gamma$ such that
\[
D_{a_i}(P') \subseteq \Gamma , \,\,\,\,\, \forall \,\, 1 \leq i \leq 3
\]
and with the quantitative information
\[
|P'| \gg  n^{1-cC'}, \,\,\,\,\,\,\rank(\Gamma) \leq C_2c \log n,
\]
where $C_2 >0$ is an absolute constant. \Cref{lem:subspaceapp} then implies that
\[
n^{1-cC'} \ll |P'| \ll T e^{C_1C_2c\log n}=Tn^{C_1C_2c},
\]
which rearranges to
\begin{equation} \label{rearrange}
T \gg n^{1-c(C'+C_1C_2)}
\end{equation}
We can set $c=\frac{\epsilon}{2(C'+C_1C_2)}$ to obtain $T \gg n^{1-\frac{\epsilon}{2}}$. Taking $n$ sufficiently large then ensures that $T \geq n^{1-\epsilon}$.

\end{proof}

 \section{The case of rich lines or circles}

 With \Cref{thm:mainnearly} now proven, the task of proving \Cref{thm:main} is reduced to the case when at least $n^{9/10}$ points of $P$ lie on a line or circle. The rich circle case was considered by Konyagin, Passant and Rudnev \cite{KPR}, where it was shown that sum-product type estimates can be applicable, and in this section we lightly adapt and generalise these observations. The following result of Elekes, Nathanson and Ruzsa \cite{ENR} connecting sum sets and convexity will be used.

 \begin{theorem} \label{thm:ENR}
    Let $I \subset \mathbb R$ be an interval. For any strictly convex or concave function $f:I \rightarrow \mathbb R$ and $A \subset \mathbb R$,
    \[
 \max \{|A-A|, |f(A)-f(A)| \} \gg |A|^{5/4}.
 \]
 \end{theorem}

 We note that quantitative improvements to this result are known, with the current record lower bound due to Bloom \cite{Bl}. Using these improved estimates would yield a small improvement for the value of $c$ in the statement of Theorem \ref{thm:main}, but we do not pursue this, and any exponent strictly greater than $1$ is sufficient for our purposes. 

 The next result was implicit in the work of Konyagin, Passant and Rudnev \cite{KPR}.

 \begin{lemma} \label{lem:richcircles}
     Let $C$ be a circle with centre $c$. Let $P$ be a set of $n$ points on $C$, let $ p \in C \setminus P$ and let $ q \in \mathbb R^2 \setminus (C \cup\{c\}) $. Then
     \[
     \max \{| \mathcal A_p(P)|,  \mathcal |\cA_q(P)|\} \gg n^{5/4}.
     \]
 \end{lemma}  

 \begin{proof}
   After a similarity transformation, we may take $C$ to be the unit circle and $p=(1,0)$. 
   
   For $r,s \in \mathbb R^2$, let $d_r(s)$ denote the oriented angle determined by $r$ and $s$ with respect to the $x$-axis. In complex notation, $d_r(s)$ denotes the argument of $s-r$.

   It will be convenient to pass to a subset of $P$ so that all the points are in the same quadrant with respect to the central point $q$. For this, we can decompose $C$ into four disjoint sets $C_1,C_2,C_3$ and $C_4$ whereby
\begin{align*}
    C_1 & = \{ s \in  C : d_q(s) \in [0, \pi/2) \},
\\ C_2 & = \{ s \in  C : d_q(s) \in [\pi/2, \pi) \},
\\ C_3 & = \{ s \in  C : d_q(s) \in [\pi, 3\pi/2) \},
\\ C_4 & = \{ s \in  C : d_q(s) \in [3\pi/2, 2\pi) \}.
\end{align*}
At least one of these sets contains a positive proportion of the elements of $P$. For simplicity, let us assume that this is the case for $C_1$, and let $P_1:= P \cap C_1$, so that $|P_1| \gg |P|$. The other three cases can be handled with some minor adaptations of the forthcoming proof. We parameterise the set $C_1$ so that
\[
   C_1= \{ (-\cos (2x), \sin (2x)) : x \in J \},
   \]
   for some set $J \subset (-\pi/2,\pi/2)$. It is possible that $J$ is not an interval (depending on the position of $q$, it may be the union of two disjoint intervals). For technical reasons, we want to restrict the domain to an interval, and so we pass to a further subset
   \[
   P_2= \{ (-\cos (2x), \sin (2x)) : x \in X \} \subseteq P_1
   \]
   where $X \subseteq I \subseteq J$ is a set with cardinality $\Omega(n)$ and $I$ is an interval.  Define
   \[
   \mathcal D_p(P_2):=\{ d_p(s) : s \in P_2 \}.
   \]
   This is essentially the same as the set $D_p(P_2)$ defined in Section \ref{sec:freiman}, up to a factor of $2$, but we use here real rather than complex notation. In particular, an important fact for us is that
   \[
   \cA_p(P_2)=(\cD_p(P_2) -\cD_p(P_2)) \cap(0, \infty),
   \]
   and in particular $|\cA_p(P_2)| \gg |\cD_p(P_2)- \cD_p(P_2)|$.
   For $x \in (-\pi/2,\pi/2)$, we have
   \[
   d_p(-\cos(2x), \sin(2x))=\arctan \left ( \frac{-\sin(2x)}{1+\cos (2x)} \right ) = \arctan \left ( \frac{-2\sin x\cos x}{2\cos^2 x} \right )= \arctan (- \tan x)=-x
   \]
   and it therefore follows that
   \[
   \cA_p(P_2) \gg |\cD_p(P_2) - \cD_p(P_2)| =  |X-X| .
   \]
   
   Next, consider the direction set $\cD_q(P_2)$. We write $q=(a,b)$, and by the assumptions of the statement we know that $a^2+b^2 \neq 1$ and at least one of $a$ and $b$ is non-zero. Since all elements of $P_2$ are above and to the right of $q$, it follows that
   \[
   \cD_q(P_2)= \{d_q( -\cos(2x),\sin(2x)) : x \in X\} = \left \{ \arctan \left ( \frac{b-\sin(2x)}{a+\cos(2x)}\right ) : x \in X \right \} .
   \]
   Define a function $f:I \rightarrow \mathbb R$ by the formula
   \[
   f(x):= \arctan \left ( \frac{b-\sin(2x)}{a+\cos(2x)}\right ).
   \]
   We claim that the second derivative $f''$ has at most two zeroes in the interval $(-\pi/2,\pi/2)$. Once this claim is established, it follows that there is a subinterval $I' \subset I$ such that $f$ is strictly convex or concave on $I'$ and $\Omega(n)$ elements of $X$ belong to $I'$. Let $X':= X \cap I'$ and let $P_3 \subseteq P_2$ be the set of points
   \[
   P_3:=  \{ (-\cos (2x), \sin (2x)) : x \in X' \}.
   \]
   Then
   \begin{align*}
    \max \{| \mathcal A_p(P)|,  \mathcal |\cA_q(P)|\} &\geq  \max \{| \mathcal A_p(P_3)|,  \mathcal |\cA_q(P_3)|\} 
    \\ & \gg \max \{ |X'-X'|, |f(X')-f(X')| \} 
    \\ &\gg |X'|^{5/4} \gg n^{5/4},
    \end{align*}
  where the penultimate inequality uses the strict convexity/concavity of $f$ and the statement of \Cref{thm:ENR}. This proves the lemma, except that it remains to verify the claim that $f''$ has at most two zeroes. A direct calculation gives
  \[
  f'(x)= \frac{2(b \sin (2x)-a \cos(2x) -1)}{(a+\cos(2x))^2+(b-\sin(2x))^2}.
  \]
  Note that the denominator above is non-zero by the assumption that $(a,b)$ is not on the unit circle. The second derivative is then given by
  \[
  f''(x)= \frac{4(a^2+b^2-1)(a \sin(2x) +b \cos(2x))}{[(a+\cos(2x))^2+(b-\sin(2x))^2]^2}.
  \]
  Since $(a,b)$ is not on the unit circle, the first factor of the numerator is a non-zero constant. The second factor $a \sin(2x) +b \cos(2x)$ has at most two zeroes in $(-\pi/2,\pi/2)$. Indeed, if $a=0$ the zeroes are precisely $x= \pm \pi/4$. If $a \neq 0$ then this rearranges to the form $\tan (2x)=-b/a$, which has two solutions in $(-\pi/2, \pi/2)$.
 \end{proof}

 The case of rich lines is handled similarly, although the analysis turns out to be a little easier.

 \begin{lemma} \label{lem:richlines}
     Let $P$ be a set of $n$ points on a line $\ell$ and let $ p, q \in \mathbb R^2 \setminus \ell $ such that $p$ and $q$ are distinct and not symmetric with respect to $\ell$. Then
     \[
     \max \{| \mathcal A_p(P)|,  |\mathcal A_q(P)|\} \gg n^{5/4}.
     \]
 \end{lemma}  

 \begin{proof}
       After a similarity transformation, we may assume without loss of generality that $\ell$ is the $x$-axis, $p=(0,1)$ and $q =(a,b)$, where $b \neq 0$. Write
     \[
     P=\{(x,0) : x \in X\}
     \]
     and define
     \[
     A= \arctan X \subseteq (-\pi/2, \pi/2).
     \]
     Then
     \[
      \mathcal A_p(P)= (A-A) \cap (0, \pi)]
     \]
     and in particular
     \[
     |\cA_p(P)| \geq \frac{|A-A|-1}{2} \gg |A-A|.
     \]
     Similarly,
     \[
      \mathcal A_q(P)= \left \{ \arctan \left (\frac{a-x}{b} \right ) - \arctan \left (\frac{a-x'}{b} \right ): x,x' \in X \right \} \cap(0, \pi).
     \]
     Define $g: (-\pi/2, \pi/2) \rightarrow \mathbb R$ by the formula
     \[
     g(t):= \arctan \left (\frac{a-\tan(t)}{b} \right ).
     \]
     Then
     \[
     \mathcal A_q(P)=(g(A)-g(A)) \cap(0,\pi)
     \]
     and so
     \[
      |\cA_q(P)|  \gg |g(A)-g(A)|.
     \]
     A direct calculation shows that $g''(t)=0$ has at most $2$ solutions with $t$ belonging to the interval  $ (-\pi/2, \pi/2)$. Indeed,
     \[
     g'(t)=-\frac{b}{(b\cos t)^2 +(a \cos t -\sin t)^2}
     \]
     and 
     \begin{align*}
     g''(t)&= \frac{2b[(1-a^2-b^2) \cos t\sin t-a(\cos^2t -\sin^2 t )]}{[(b\cos t)^2 +(a \cos t -\sin t)^2]^2}
     \\& =\frac{2b[\frac{1}{2} (1-a^2-b^2)\sin (2t) -a\cos(2t) ]}{[(b\cos t)^2 +(a \cos t -\sin t)^2]^2},
     \end{align*}
     where the last equality follows from the double angle formula. Note that $g'$ and $g''$ are well-defined; this follows from the assumption that $b \neq 0$. Any solution to $g''(t)=0$ must satisfy
     \begin{equation} \label{tanred}
     \frac{1}{2} (1-a^2-b^2) \sin (2t) =a\cos(2t).
     \end{equation}
     We can rule out the case that both coefficients $\frac{1}{2} (1-a^2-b^2)$ and $a$ are zero, since this implies $(a,b)=(0,1)$ (which is excluded by the assumption that $p$ and $q$ are distinct) or $(a,b)=(0,-1)$ (which is excluded by the assumption that $p$ and $q$ are not symmetric with respect to $\ell$). If $\frac{1}{2} (1-a^2-b^2)=0$ then \eqref{tanred} has $2$ zeroes in the given range. Otherwise, \eqref{tanred} can be rearranged as 
     \[
     \tan(2t)= \frac{2a}{1-a^2-b^2},
     \]
     which has at most two zeroes in the given range.
      
     In particular, there exists an interval $I$ such that $g$ is strictly convex or concave on $I$ and $A' := A \cap I$ satisfies $|A'| \gg |A|$. \Cref{thm:ENR} then implies that
     \[
     \max \{| \mathcal A_p(P)|,  \mathcal A_q(P)|\} \gg \max \{ |A'-A'|, |g(A')-g(A')| \} \gg n^{5/4}. 
     \]
    
 \end{proof}

\section{Proof of main result}

We are now ready to prove \Cref{thm:main}.

\begin{proof}[Proof of \Cref{thm:main}]
Let $n$ be sufficiently large and let $P\subseteq \mathbb R^2$ be a point set which does not take one of the five forbidden forms in the statement of \Cref{thm:main}. We will show that there exists $q \in P$ such that $|\mathcal A_q(P)| \geq n^{1+c}$ for some absolute constant $c>0$.

Suppose first that $P$ contains at most $n^{9/10}$ points on a line or circle. Apply \Cref{thm:mainnearly} with $\epsilon=1/10$. Then there exists $q \in P$ such that $|\mathcal A_q(P)| \geq n^{1+c}$ for some $c>0$. This value of $c$ is now fixed for the rest of the proof. We can assume that $c < 1/10$; the reality is much worse.

Suppose now that $P$ contains at least $n^{9/10}$ points on a circle $C$ with centre $c$. Since $P$ does not take the fourth or fifth forbidden form, there exists $q \in P \setminus (C \cup \{c\})$. Let $p$ be an arbitrary point of $P \cap C$ and set $P'=P \setminus \{p\}$. Then \Cref{lem:richcircles} implies that
\[
 \max \{| \mathcal A_p(P' \cap C)|,  |\mathcal A_q(P' \cap C)|\} \gg |P' \cap C|^{5/4} \gg n^{9/8} \geq n^{1+c},
\]
where the last inequality is valid provided that $n$ is sufficiently large. 

The case when $P$ contains at least $n^{9/10}$ points on a line $\ell$ is handled similarly. Since $P$ does not take any of the first three forbidden forms, there exist two points $p,q \in P \setminus \ell$ such that $p$ and $q$ are not symmetric with respect to $\ell$. \Cref{lem:richlines} then gives
\[
 \max \{| \mathcal A_p(P \cap \ell)|,  |\mathcal A_q(P \cap \ell)|\} \gg |P \cap \ell|^{5/4} \gg n^{9/8} \geq n^{1+c},
\]
which concludes the proof of \Cref{thm:main}.
    
\end{proof}

\section{Classifying point sets defining minimal angles}
In this section we use \Cref{thm:main} to prove \Cref{thm:classification}. The statement of \Cref{thm:classification} is repeated below for the convenience of the reader.

\begin{theorem}\label{thm:classificationagain}
    For $n$ sufficiently large, every non-collinear point set $P \subseteq \mathbb R^2$ of size $n$ determines at least $n-2$ distinct angles. Furthermore, the only point configuration of $n$ non-collinear points which achieves exactly $n-2$ distinct angles is the regular $n$-gon.
\end{theorem}
\begin{proof}
    Firstly, assume that $P$ is not of one of the five forbidden forms in the statement of \Cref{thm:main}.
     Then, applying \Cref{thm:main}, the point set $P$ determines at least $n^{1+c}$ angles - in this case we are done. Since we assume that $P$ is non-collinear, the four cases remaining are:

    \begin{enumerate}
        \item $n-2$ collinear points, with the two remaining points off the rich line.
        \item $n-1$ collinear points, with the final point off the rich line.
        \item $n-1$ points co-circular, with the final point being the centre of the circle.
        \item All points of $P$ are co-circular.
    \end{enumerate}

    It is not too hard to see that any set of $k$ points on a circle $C$ must always determine at least $k-2$ distinct angles, by fixing two points $p,q$ which are adjacent on the circle and considering the angles $\cA(p,q,p')$ with $p' \in C \setminus \{p,q\}$. Similarly, any set of $k$ points on a line $\ell$ and a single point $q$ not lying on $\ell$ must define $k+1$ distinct angles, since fixing an extremal point $p \in \ell$ and considering the angles $A(p,q,p')$ as $p'$ passes along the line gives $k-1$ distinct angles, but we also get the angles $0$ and $\pi$ by taking three points on $\ell$. Therefore we have at least $k+1$ distinct angles. These two observations immediately deal with cases $(1),(2)$ and $(4)$, and furthermore show that cases $(1)$ and $(2)$ determine strictly more than $n-2$ angles.

    We now deal with case $(3)$. Let $C$ denote the rich circle and let $c$ denote its centre. Fix two adjacent points $q,p$ on $C$ and consider the angles $\cA(p,q,p')$ for $p' \in (P \cap C)\setminus \{p,q\}$. These points lying on the circle yield at least $n-3$ distinct angles of this form, and this does not include either of the angles $0$ or $\pi$. Now assume that $P$ contains two antipodal points $a_1,a_2$ of the circle - in this case, the triple $a_1,c,a_2$ is collinear, and the two angles $0$ and $\pi$ have been defined giving at least $n-1$ distinct angles. We can therefore henceforth assume that no two points of $P$ lying on the circle are antipodal. 
    
    Fix a point $p \in P \cap C$, and consider the $n-1$ distinct line segments connecting $p$ with the other points of $P$. Letting $q$ be an adjacent point to $p$ on the circle, these line segments give rise to $n-2$ angles of the form $\cA(q,p,p')$ for $p' \in P \setminus \{p,q\}$. Let $B$ denote the set of these angles, union with $\{0\}$. The set of angles of the form $\cA(r,p,s)$ for $r,s \in P \setminus \{p\}$ is the positive part of the difference set $B-B$. Since $B$ has size $n-1$, we are done unless $B$ forms an arithmetic progression - indeed, we have
    $$|(B-B) \cap \mathbb R^+| = \frac{1}{2}\left(|B-B| - 1 \right) \geq \frac{1}{2}(2n-4) = n-2,$$
    and equality holds if and only if $B$ is an arithmetic progression. We then write $B$ as
    $$B = \{ i \theta : 0 \leq i \leq n-2 \},$$
    and we label the points of $(P \cap C) \setminus \{p,q\}$ in cyclic order as $p_1, \dots, p_{n-3}$.
    So far, we have found $n-2$ distinct angles defined by $P$ - we are now searching for one more. Let $p_m$ denote one of the two points of $P \cap C$ closest to the antipodal point of $p$. The triangle $pcp_m$ is an isosceles triangle with two occurrences of the angle $\theta$, and a central angle of $\pi- 2\theta$. If this angle of $\pi - 2\theta$ does not occur in the positive difference set of $B$, we are done as we have found a new angle. Therefore we assume it does occur in the positive difference set of $B$ - that is, 
    $$\pi - 2\theta = k \theta, \quad k \in \{1,...,n-2\}.$$
    We then have $\pi = (k+2) \theta$, however if we combine this with the knowledge that all angles in $\cA_p(P)$ are less than $\pi$, we have
    $$(n-2)\theta < \pi = (k+2)\theta \implies k \geq n-3.$$
    Therefore the only options are $k=n-3,n-2$. The situation so far is depicted in the figure below.
\begin{figure}[h!]
\centering
\begin{tikzpicture}
\coordinate (p) at (-4,0);
\coordinate (q) at (-3.5,1.95);
\coordinate (p1) at (-1.8,3.55);
\coordinate (p2) at (0,4);
\coordinate (t1) at (-4,1);
\coordinate (t2) at (-4,-1);
\coordinate (r) at (0,0);
\coordinate (pm) at (3.8,1.2);
\coordinate (p_{n-1}) at (-3.5,-1.95);

 \filldraw[black] (1.2,3) circle (1pt);
  \filldraw[black] (0.8,3.2) circle (1pt);
   \filldraw[black] (0.4,3.4) circle (1pt);

    \draw (0,0) circle (4);
    \draw (-4,-4) -- (-4,4);
    \draw (p) -- (-3.5,1.95);
    \draw (p) -- (p1);
    \draw (p) -- (p_{n-1});
    \draw (p) -- (p2);
    \draw (p) -- (r);
    \draw (p) -- (pm);
    \draw (r) -- (pm);
    \filldraw[black] (p) circle (1pt) node[left] {$p$};
    \filldraw[black] (p_{n-1}) circle (1pt) node[right] {$p_{n-3}$};
    \filldraw[black] (p1) circle (1pt) node[left] {$p_1$};
    \filldraw[black] (p2) circle (1pt) node[above] {$p_2$};
    \filldraw[black] (-3.5,1.95) circle (1pt) node[right] {$q$};
    \filldraw[black] (r) circle (1pt) node[below] {$c$};
    \filldraw[black] (pm) circle (1pt) node[right] {$p_m$};
    \pic[
    draw,
    "$\phi_1$",
    angle radius=2cm,
    angle eccentricity=1.2
] {angle=q--p--t1};
\pic[
    draw,
    "$\phi_2$",
    angle radius=2cm,
    angle eccentricity=1.2
] {angle=t2--p--p_{n-1}};
\pic[
    draw,
    "$\theta$",
    angle radius=0.6cm,
    angle eccentricity=2
] {angle=p1--p--q};
\pic[
    draw,
    "$\theta$",
    angle radius=0.6cm,
    angle eccentricity=2
] {angle=p2--p--p1};
\pic[
    draw,
    "$\theta$",
    angle radius=1cm,
    angle eccentricity=2
] {angle=r--p--pm};
\pic[
    draw,
    "$\pi - 2\theta$",
    angle radius=0.2cm,
    angle eccentricity=2
] {angle=pm--r--p};
\pic[
    draw,
    "$\theta$",
    angle radius=1cm,
    angle eccentricity=2
] {angle=p--pm--r};
    
\end{tikzpicture}
\end{figure}
\textbf{Case 1 - $(n-1)\theta = \pi$.} In this case, we have that the two tangent angles $\phi_1$ and $\phi_2$ sum to $\theta$; indeed, summing all of the angles defined on the right side of the tangent line at $p$, we get 
$$\phi_1 + (n-2)\theta + \phi_2 = \pi = (n-1)\theta \implies \phi_1 + \phi_2 = \theta.$$
 In particular, both $\phi_1,\phi_2$ are strictly less than $\theta$. By the inscribed angle theorem (see the left image below), the angle $A(p,c,q)$ is equal to $2\phi_1$. Therefore we find a new angle, and are done, unless $\phi_1 = \frac{\theta}{2}$. But if this happens, then by a second inscribed angle theorem (see the right image below), we have that $A(p,p_1,c) = \frac{\theta}{2}$, and we have again found a new angle. In either case, we are done. The two inscribed angle theorems we have used are depicted below.

\begin{minipage}{0.4 \linewidth}
\centering
\begin{tikzpicture}[scale=1]
\coordinate (p) at (-2,0);
\coordinate (q) at (-3.5,1.95);
\coordinate (p1) at (-1,1.75);
\coordinate (p2) at (0,4);
\coordinate (t1) at (-2,1);
\coordinate (t2) at (-4,-1);
\coordinate (O) at (0,0);
\coordinate (pm) at (3.8,1.2);
\coordinate (p_{n-1}) at (-3.5,-1.95);

    \draw (0,0) circle (2);
    \draw (-2,-2) -- (-2,2);
    \draw (p) -- (r);
    \draw (p) -- (p1);
    \draw (p1) -- (r);
    \filldraw[black] (O) circle (1pt) node[right] {$c$};
    \filldraw[black] (p) circle (1pt) node[left] {$p$};
    \filldraw[black] (p1) circle (1pt) node[above] {$q$};
    \pic[
    draw,
    "$\phi_1$",
    angle radius=1.5cm,
    angle eccentricity=1.2
] {angle=p1--p--t1};
\pic[
    draw,
    "$2 \phi_1$",
    angle radius=0.4cm,
    angle eccentricity=2
] {angle=p1--O--p};
    
\end{tikzpicture}

Angle Theorem 1
\end{minipage}
\begin{minipage}{0.5 \linewidth}
\centering
\begin{tikzpicture}[scale=1]
\coordinate (p) at (-2,0);
\coordinate (q) at (-1,1.75);
\coordinate (p1) at (0.8,1.84);
\coordinate (p2) at (0,4);
\coordinate (t1) at (-2,1);
\coordinate (t2) at (-4,-1);
\coordinate (O) at (0,0);
\coordinate (pm) at (3.8,1.2);
\coordinate (p_{n-1}) at (-3.5,-1.95);

    \draw (0,0) circle (2);
    \draw (p) -- (r);
    \draw (q) -- (r);
    \draw (p) -- (p1);
    \draw (p) -- (q);
    \draw (p1) -- (q);
    \filldraw[black] (O) circle (1pt) node[right] {$c$};
    \filldraw[black] (p) circle (1pt) node[left] {$p$};
    \filldraw[black] (q) circle (1pt) node[above] {$q$};
    \filldraw[black] (p1) circle (1pt) node[above] {$p_1$};
\pic[
    draw,
    "$\theta$",
    angle radius=0.4cm,
    angle eccentricity=2
] {angle=q--O--p};
\pic[
    draw,
    "$\frac{\theta}{2}$",
    angle radius=0.6cm,
    angle eccentricity=1.6
] {angle=q--p1--p};
\end{tikzpicture}

Angle Theorem 2
\end{minipage}

\textbf{Case 2 - $n\theta = \pi$.} In this case, we can repeat the same argument above summing the angles on the right side of the tangent line at $p$ to obtain
$$\phi_1 + (n-2)\theta + \phi_2 = \pi = n \theta \implies \phi_1 + \phi_2 = 2\theta.$$
Assume WLOG that $\phi_1 \leq \theta$. We repeat the rest of the argument above - from the first inscribed angle theorem, we must have that $2\phi_1 \in \{\theta,2\theta \}$ or we are done. If we have $\phi_1 = \frac{\theta}{2}$, then the second inscribed angle theorem finds an angle $\frac{\theta}{2}$ defined by $P$ and we are done. Therefore we must have $\phi_1 = \theta = \phi_2$.

Note that if $n$ is even, then since $n\theta = \pi$, the angle $\frac{\pi}{2}$ appears in the angles defined with a centre at $p$. In fact since the angles form an arithmetic progression, the angle $\pi/2$ appears many times - in particular, of the form $A(p_i, p, p_j)$ for $p_i,p_j \in P \cap C$. However having such a right angle inscribed in the circle implies that $p_i$ and $p_j$ are antipodal, which we have assumed is not the case (else we define the angles $0,
\pi$). Therefore if $n$ is even, we are done.

Now assume that $n$ is odd. Excluding the points $p$ and $c$, there are $n-2$ remaining points of $P$ which lie on the circle. Since no two points of $P \cap C$ are antipodal, the antipodal point of $p$ does not lie in $P$. Therefore each point of $(P\setminus \{p\}) \cap C $ is either above the line connecting $p$ and $c$, or below it. Furthermore, from the structure of the angles with $\phi_1 = \phi_2 = \theta$, any point $p_i \in P \cap C$ must also have its reflection in the line connecting $p$ and $c$, call it $\overline{p_i}$, lying within $P$. But then there must be an even number of points, since we have just defined a precise pairing of the points. This is a contradiction to $n$ being odd. Therefore case (3) also determines at least $n-1$ angles.

So far, we have proved that if $P$ is a set of $n$ points determining $n-2$ angles, then all points of $P$ must be on a circle. We finish the proof by showing that these points must form a regular $n-$gon. Label the points of $P$ clockwise around the circle by $p_1,p_2,...,p_n$. Fixing any two adjacent points $p_i,p_{i+1}$, we obtain a set of $n-2$ distinct angles 
\[B_i := \{\cA(p_{i+1},p_{i},q) : q \in P \setminus \{p_{i},p_{i+1}\}\}.
\]
Just as in the previous section, the set of pinned angles at $p_{i}$ defined by $P$ is the positive difference set $\left((B_i \cup \{0\}) - (B_i \cup \{0\})\right) \cap \mathbb R^+$. Again, the only way for this to be size $n-2$ is for $B_i$ itself to be an arithmetic progression. Furthermore, this arithmetic progression must be the same for each $i$, since otherwise we have more angles. In particular, for each $i,j$ we have $\cA(p_{i+1},p_i,p_{i+2}) = \cA(p_{j+1},p_j,p_{j+2})$; this is the smallest angles defined from two neighbouring points $p_i,p_{i+1}$. Applying Angle Theorem 2 from above then yields that all internal angles of this configuration are the same, so that $P$ is indeed a regular $n-$gon. This concludes the proof of \Cref{thm:classification}.
\end{proof}

\section{Construction with two bad pins}

If we restrict our attention to the case when $P$ is in general position, then the analysis in the proof of \Cref{thm:main} becomes slightly easier. Since no line contains more than $2$ points of $P$, \Cref{lem:preprocessing} is trivially true, and in fact the statement is valid for \textit{any} choice of three distinct points $p_1,p_2,p_3 \in P$. Our argument then gives the following result.

\begin{theorem} \label{thm:genpos}
   Let $n$ be sufficiently large and let $P \subset \mathbb R^2$ be a set of $n$ points in general position. Then there is an absolute constant $c>0$ such that the inequality
   \[
   |\cA_p(P)| \geq n^{1+c}
   \]
   is valid for all but at most two elements $p \in P$.
\end{theorem}

In other words, a point set in general position can have at most two bad pins. We now show that this statement is optimal, in the sense that there really do exist general position point sets with two bad pins.

\begin{construction}\label{construction}
    There exists a point set $P \subseteq \mathbb R^2$ of size $n$ which is in general position, such that there are two points $p,q \in P$ which each define only $n-2$ pinned angles.
\end{construction}

\begin{proof}
    We will construct the point set $P$ using sets of evenly spaced lines through two fixed points, which will be the two `bad' pins. Let $p = (0,0)$ and $q = (1,0)$, and let $\theta$ be a very small angle (of the order, say, $\frac{\pi}{10n}$). Consider the collection of $n-1$ lines passing through $p$ of the form
    $$\ell_i := \left\{(x,y) : y = \tan(i\theta) x \right\}, \quad i=0,...,n-2.$$
    These lines are spaced evenly with angle $\theta$ between consecutive lines. In a similar way, we define a set of $n-1$ lines passing through $q.$
    $$m_i := \left\{(x,y) : y = - \tan(2i\theta) (x-1) \right\}, \quad i=0,...,n-2.$$
    We now define the point set $P$ as the intersection points $\ell_i \cap m_i$ for $i=1,...,n-2$, together with the two points $p$ and $q$. We first claim that the point set $P \setminus \{(1,0)\}$ is contained  within the conic $C$ given by the equation $y^2 = 3x^2 - 2x$. Indeed, let $(x_i,y_i) = p_i:= \ell_i \cap m_i$. This point satisfies
    $$\frac{y_i}{x_i} = \tan(i\theta), \quad  \frac{y_i }{1 - x_i} = \tan(2i\theta) = \frac{2 \tan(i\theta)}{1 - \tan^2(i\theta)}.$$
    Substituting in, we obtain
    $$\frac{y_i}{1-x_i}= \frac{2\left(\frac{y_i}{x_i}\right)}{1- \left(\frac{y_i}{x_i}\right)^2} \implies y_i^2 = 3x_i^2 -2x_i.$$
    It is clear that $p$ lies on $C$, finishing the proof of our claim. Given that $C$ is an irreducible conic, in particular a hyperbola, it is immediate that no three points of $P\setminus \{q\}$ are collinear. We now need to show that no four points of $P$ are co-circular, and that $q$ is not collinear with two other points of $P$. This is best explained geometrically - see the \Cref{fig:2badpins} below for the curve $C$ and the points $P$.
    
\begin{figure}[h]
    \centering
    \includegraphics[width=0.8\linewidth]{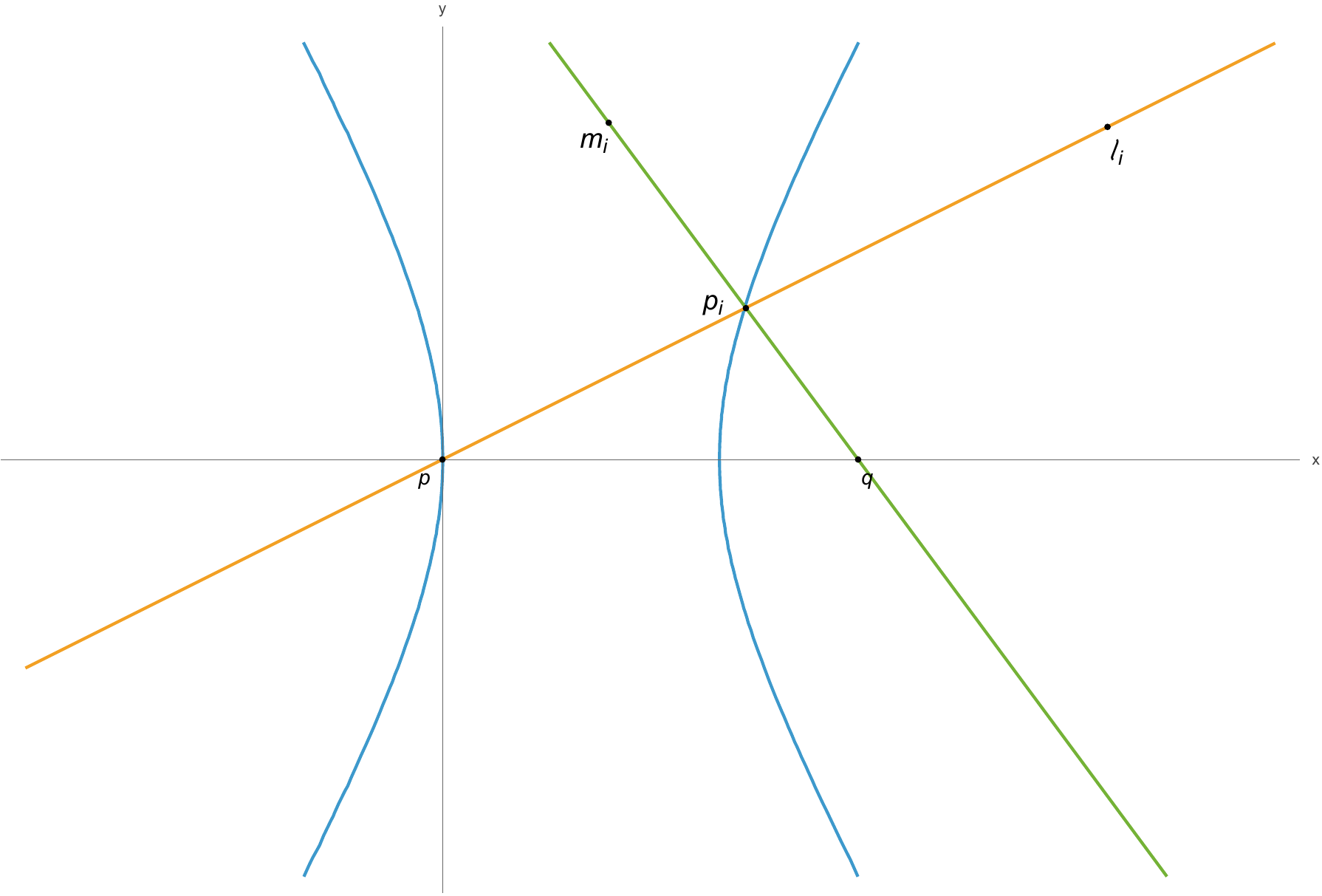}
    \caption{The points $p$ and $q$, the curve $C$, and an example point $p_i$.}
    \label{fig:2badpins}
\end{figure}

    Firstly, notice that since the lines $\ell_i$ all have positive slope and the lines $m_i$ all have negative slope, all intersection points occur in the positive quadrant of $\mathbb R^2$. Furthermore, notice that in the positive quadrant the curve $C$ is given by an implicit function $y = f(x)$ with $y'' <0$; that is, the slope of the curve is always decreasing. With this we can see that any line passing through $q$ and a point $p_i$ (or $p$) cannot intersect $C$ again in the positive quadrant; indeed since this line passes through $q$, it is below the graph of the function when it enters the positive quadrant, and can then only have a single intersection point with $y=f(x)$ - past this point the line lies strictly above the graph of the function since the function has decreasing derivative. Therefore we have proved that no three points of $P$ are collinear.

    We now wish to prove that no four points of $P$ are cocircular. We do this by representing the curve $C$ in the positive quadrant as the function
    $$y = f(x) := \sqrt{3x^2 - 2x}, \quad x > 2/3.$$
    All points of $P$ apart from $p$ and $q$ lie on the graph of this function. We prove first that no circle intersects the graph of this function in four points. Suppose we take a circle
    $$\gamma := \{(x,y) : (x-a)^2 + (y-b)^2 = r\}.$$
    Any point $(x,y)$ lying on both $\gamma$ and on the graph of $f$ must satisfy
    $$(x-a)^2 + (f(x)-b)^2 = r \implies h(x):= 4x^2 - (2a+2)x - 2bf(x) +a^2 + b^2 - r=0.$$
    Note that if $b=0$ then $h$ is quadratic in $x$ and so has at most two zeros, and we are done. Therefore we can assume that $b \neq 0$. Note that after differentiating three times, we have
    $$h'''(x) = -2bf'''(x) = \frac{-2b(9x-3)}{x^{5/2}(3x-2)^{5/2}} \text{ on } x > 2/3.$$
    Depending on the sign of $b$, the third derivative $h'''(x)$ is either always positive or always negative on $x>2/3$, since every factor in the expression above has constant sign in this interval. However by Rolle's theorem (applied three times) a function with strictly positive, (or strictly negative) third derivative on $x > 2/3$ cannot have four zeros within $(2/3,\infty)$. Therefore no four points of $P - \{p,q\}$ are cocircular.

    Next, we prove that the point $q$ is not cocircular with three other points of $P \setminus \{p\}$. To do this, we begin by writing a general circle $\gamma$ passing through $q = (1,0)$ as
    $$(x-a)^2 + (y-b)^2 = (a-1)^2 + b^2.$$
    Again using that all points of $P \setminus \{p,q\}$ lie on the graph of the function $f(x) = \sqrt{3x^2-2x}$ with $x > 2/3$, we see that any point lying on both $\gamma$ and the graph of $f$ must satisfy
    $$  4x^2 - (2a+2)x +2a-1 = 2bf(x) \implies g(x):= \frac{4x^2 - (2a+2)x +2a-1}{f(x)} = 2b.$$
    We would therefore like to prove that the function $g(x)$ cannot take the same function value (namely $2b$) more than twice on the region $x > 2/3$. We again do this by considering derivatives - we have
    $$g'(x) = \frac{(2x-1)(6x^2 - 3x +1-2a)}{x^{3/2}(3x-2)^{3/2}}.$$
    In the region $x > 2/3$, the only factor of $g'(x)$ which can change sign is $6x^2-3x+1-2a$. This quadratic curve is strictly increasing on $x>2/3,$ and therefore can only change sign once. Therefore $g(x)$ can only take the value $2b$ at most twice in this region, proving that $q$ is not cocircular with three other points of $P \setminus \{p\}$.

    Next we must prove that $p$ is not cocircular with three other points of $P \setminus \{q\}$. This is essentially the same as the previous proof, so we do not repeat it. Finally, we must show that there is no circle through $p$ and $q$ which passes through two other points of $P$. This is also similar to the previous cases - consider a circle passing through $(0,0)$ and $(1,0)$ given by the equation
    $$(x-1/2)^2 + (y-b)^2 = 1/4 + b^2.$$
    Intersecting this equation with $y = f(x)$ as above, one finds an equation 
    $$g(x) := \frac{4x^2 - 3x}{f(x)} = 2b,$$
    whose solutions for $x>2/3$ correspond to intersection points of the circle with $y=f(x)$, but where the first derivative 
    $$g'(x) = \frac{3x(2x-1)^2}{x^{1/2}(3x-2)^{3/2}}$$
    is strictly positive for $x>2/3$ - and therefore $g(x)$ can only take value $2b$ at most once, so that the circle can only contain at most one more point $p_i$ from $P$. This concludes the proof of \Cref{construction}, since it is clear that each of $p$ and $q$ determine only $n-2$ angles to the rest of $P$.
    
    \end{proof}

    \section*{AI involvement in this research}
    The research presented in this paper has benefitted from AI collaboration, namely ChatGPT 5.6 Sol. The main contribution from AI was that the subspace theorem can be used in the argument, leading to a resolution of the problem for the general position case. Apart from this, all other ideas of the proof were created by the authors. Everything in this article is human written and verified.

     We would also like to take some time here to try to give some credit to those works which may have inspired the AI contribution to this work; a potentially thorny issue arising in this new age of AI and mathematics. The work of Chang \cite{Ch} in relating the subspace theorem to the sum-product problem is likely to have been influential here. A paper of Schwartz \cite{Sc} used the subspace theorem in the discrete geometric context, using it to bound the maximum number of occurrences of a unit distances determined by $P \subset \mathbb R^2 \cong \mathbb C$ if the point set is contained in a low rank multiplicative subgroup.

    \section*{Acknowledgements} We are very grateful to Misha Rudnev for introducing us to the problem of angles determined by point sets in general position, and for many interesting discussions we have had about it over several years, as well as for giving helpful feedback on an earlier draft of this paper. We also thank Thomas Bloom, Jakob Führer, Michalis Kokkinos, Sam Mansfield, Jonathan Passant and Micha Sharir for the helpful conversations we have had with them concerning this problem. Krishnendu Bhowmick was funded by a grant from the United States-Israel Binational Science Foundation (BSF), Jerusalem, Israel, and the United States National Science Foundation (NSF). Oliver Roche-Newton and Audie Warren were partially supported by the Austrian Science Fund (FWF) project PAT2559123.

\bibliography{decompose}
\bibliographystyle{plain}

\end{document}